\documentclass[12pt]{amsart}
\usepackage{epsfig,color}

\makeatother

\theoremstyle{definition}

\def\fnum{equation} 
\newtheorem{Thm}[\fnum]{Theorem}
\newtheorem{Cor}[\fnum]{Corollary}

\newtheorem{Lem}[\fnum]{Lemma}
\newtheorem{Con}[\fnum]{Conjecture}

\newtheorem{Rem}[\fnum]{Remark}
\newtheorem{Pro}[\fnum]{Proposition}

\numberwithin{equation}{section}
\newcommand{\nn}{{\bf{n}}}

\newcommand{\reg}{{\text{Reg}}}
\newcommand{\sing}{{\text{Sing}}}

\def\RR{{\bold R}}

\def\SS{{\bold S}}

\newcommand{\dv}{{\text {div}}}
\newcommand{\e}{{\text {e}}}

\newcommand{\cL}{{\mathcal{L}}}

\newcommand{\eqr}[1]{(\ref{#1})}

\title[Minimal entropy]{The round sphere minimizes entropy among closed self-shrinkers}

\author{Tobias Holck Colding}%
\address{MIT, Dept. of Math.\\
77 Massachusetts Avenue, Cambridge, MA 02139-4307.}
\author{Tom Ilmanen}%
\address{Departement Mathematik, ETH Zentrum, 8092 Z\"urich, Switzerland.}
\author{William P. Minicozzi II}%
\address{Johns Hopkins University\\
Dept. of Math.\\
3400 N. Charles St.\\
Baltimore, MD 21218.}
\author{Brian White}%
\address{Stanford University, Dept. of Math.\\
bldg 380, Stanford, California 94305.}

\thanks{The first and third authors
were partially supported by NSF Grant DMS  
11040934, DMS
0906233,  and NSF FRG grants DMS 
 0854774 and DMS 0853501. The fourth author was partially supported by NSF grant DMS 1105330.}

\email{colding@math.mit.edu, tom.ilmanen@math.ethz.ch, minicozz@math.jhu.edu, and white@math.stanford.edu.}

\begin{document}

\maketitle

\begin{abstract}
The entropy of a hypersurface is  a geometric invariant that measures   complexity and is invariant under rigid motions and dilations.  It 
is given by  the supremum over all Gaussian integrals with varying centers and scales.  It is monotone under mean curvature flow, 
thus giving a Lyapunov functional.   Therefore,  the entropy of the initial hypersurface   bounds the entropy at  all future singularities.  
We  show here that not only does the round sphere have the lowest entropy of any  closed singularity, but there is a gap
to  the second  lowest.
\end{abstract}

\section{Introduction}

The $F$-functional of a hypersurface $\Sigma$ of Euclidean space $\RR^{n+1}$ is the   Gaussian surface area\footnote{Gaussian surface area has also been studied in convex geometry (see \cite{B} and \cite{Na}) and in theoretical computer science (see \cite{K} and \cite{KDS}).}
\begin{equation}
F(\Sigma)= \left( 4 \, \pi \right)^{ - \frac{n}{2} } \, \int_{\Sigma}\e^{-\frac{|x|^2}{4}}\, , 
\end{equation}
whereas the Gaussian entropy is the supremum over all Gaussian surface areas given by 
\begin{equation}
\lambda (\Sigma)=\sup  \, \,  \left( 4 \, \pi \, t_0 \right)^{ - \frac{n}{2} } \, \int_{\Sigma}\e^{-\frac{|x-x_0|^2}{4t_0}}\, .
\end{equation}
Here the supremum is taking over all $t_0>0$ and $x_0\in \RR^{n+1}$.  
Entropy is invariant under   rigid motions and dilations.

Mean curvature flow (``MCF'') is an evolution equation  where a one-parameter family of hypersurfaces $M_t \subset \RR^{n+1}$
evolves over time to minimize  volume, satisfying the equation
\begin{equation}
	\left( \partial_t x \right)^{\perp} =  -H\,\nn \, .
\end{equation}
Here    $H = \dv_{\Sigma} \,  \nn$ is the mean curvature, $\nn$ is the outward pointing unit normal and $x$ is the position vector.   
As a consequence of  Huisken's monotonicity,  entropy is monotone nonincreasing under  MCF.

A  hypersurface $\Sigma \subset \RR^{n+1}$ is    a  {\emph{self-shrinker}} if it is the $t=-1$ time-slice of a  MCF moving by rescalings, i.e., where   $\Sigma_t \equiv \sqrt{-t} \, \Sigma$ is a MCF; see \cite{A}, \cite{Ch}, \cite{KKM}, \cite{M} and \cite{N} for examples.  This is easily seen to be equivalent to the equation
\begin{equation}	\label{e:ss}
	H = \frac{ \langle x , \nn \rangle }{2} \, .
\end{equation}
Since \eqr{e:ss} is the Euler-Lagrange equation for $F$, 
 self-shrinkers  are critical points for  $F$. 

Under MCF, every closed hypersurface becomes singular and the main problem is to understand the singularities. 
 By Huisken's monotonicity and an argument of \cite{I1}, \cite{W2}, blow-ups of singularities of a MCF are self-shrinkers.

By section $7$ in \cite{CM1},  the entropy of a self-shrinker   is equal to the value of $F$ and, thus,
 no supremum is needed.  In \cite{St}, Stone computed the $F$ functional, and therefore the entropy,
 for generalized cylinders
$\SS^{k} \times \RR^{n-k}$. He showed that $\lambda (\SS^n)$ is decreasing in $n$ and
\begin{align}
	\lambda (\SS^1) = \sqrt{ \frac{2\pi}{\e} } \approx 1.5203 > \lambda (\SS^2) = \frac{4}{\e} \approx 1.4715 >
	\lambda (\SS^3) > \dots > 1 = \lambda (\RR^n) \, . 
\end{align}
Moreover, a simple computation shows that $\lambda (\Sigma \times \RR) = \lambda (\Sigma)$.

  \vskip2mm
 It follows from Brakke's theorem, \cite{Br}, that $\RR^n$ has the least entropy of any self-shrinker and, in fact, there is a gap to the next lowest.
Our main result is that the round sphere has the least entropy of any {\emph{closed}} self-shrinker.  

\begin{Thm}   \label{t:minentropy}
Given $n$, there exists $\epsilon =\epsilon (n) > 0$ so that if $\Sigma\subset \RR^{n+1}$ is a closed self-shrinker not equal to the round sphere, then $\lambda (\Sigma)\geq \lambda (\SS^n)+\epsilon$.  
Moreover, if $\lambda (\Sigma)\leq \min \{\lambda (\SS^{n-1}),\frac{3}{2}\}$, then 
$\Sigma$ is diffeomorphic to $\SS^n$.{\footnote{If $n > 2$, then $\lambda (\SS^{n-1}) < \frac{3}{2}$ and the minimum is unnecessary.}}
\end{Thm}

  Theorem \ref{t:minentropy} is suggested by the dynamical approach to MCF of
   \cite{CM1} and \cite{CM2}.   The  idea is that MCF starting at  a  closed $M$ becomes singular, the corresponding self-shrinker has lower entropy and,
 by  \cite{CM1},   the only   self-shrinkers that cannot be perturbed away are   $\SS^{n-k} \times \RR^k$ and   $
 \lambda (\SS^{n-k} \times \RR^k) \geq \lambda (\SS^n)$.    However, we are not  able to make this approach rigorous.  One of the difficulties is that if the self-shrinker is non-compact and we perturb it, then {\it a priori} it may flow smoothly without ever 
 becoming singular.{\footnote{Theorem $0.10$ in \cite{CM1} is proven using an argument along these lines.}}  To overcome this, we combine results from  \cite{CM1} and \cite{IW}.

The dynamical picture also suggests two closely related conjectures; the first is for any closed hypersurface and the second is for   self-shrinkers:

\begin{Con}	\label{c:minentropy}
Theorem \ref{t:minentropy} holds with $\epsilon = 0$  for any closed hypersurface $M^n$ 
with $n\leq 6$.
\end{Con}

\begin{Con}	\label{c:c2}
Theorem \ref{t:minentropy} holds   for any {\emph{non-flat}} self-shrinker $\Sigma^n \subset \RR^{n+1}$ 
with $n\leq 6$.

\end{Con}

Both conjectures are true for curves, i.e., when $n=1$.  The first conjecture
follows for curves by combining Grayson's theorem, \cite{G} (cf. \cite{GaHa}), and the monotonicity of $\lambda$ under curve shortening flow.  The second conjecture follows for curves from the classification of self-shrinkers by Abresch and Langer, \cite{AbL}.

Conjecture \ref{c:minentropy} would allow us to carry out the outline above to show that any closed hypersurface has entropy at least that of the sphere, proving Conjecture \ref{c:c2}.

Furthermore, one could  ask which self-shrinker has the third least entropy, etc.  
It is easy to see that the entropy of the ``Simons cone'' over $\SS^k \times \SS^k$ in $\RR^{2k+2}$ is asymptotic to $\sqrt{2}$  as $k\to \infty$, which is also the limit of
$\lambda (\SS^{2k+1})$.  Thus, as the dimension increases, the Simons cones have lower entropy than some of the generalized cylinders.
For example,  the cone over $\SS^2 \times \SS^2$ has  entropy $\frac{3}{2}  < \lambda (\SS^1 \times \RR^4)$.  In other words, already for $n=5$, $\SS^k \times
\RR^{n-k}$ is not a complete list of the lowest entropy self-shrinkers.

\vskip2mm
It is easy to see that if an immersed hypersurface has entropy strictly less than two, then it is embedded, hence,   we will always assume embeddness below.

\section{Perturbing a self-shrinker}

Throughout this paper, unless otherwise mentioned, $L$ will be the second variation operator for the $F$ functional from section $4$ of \cite{CM1}, that is, 
\begin{equation}
	L = \cL + |A|^2 + \frac{1}{2} = \Delta - \langle \frac{x}{2} , \nabla \cdot \rangle + |A|^2 + \frac{1}{2}\, ,
	\end{equation}
where $A$ is the second fundamental form and the second equality defines the operator $\cL$.

\vskip2mm
The first step is to  perturb  a closed self-shrinker inside itself, while reducing the entropy and making the self-shrinker version of mean curvature positive.  This uses the classification of stable self-shrinkers from \cite{CM1}.

\begin{Lem}	\label{l:perturb}
If $\Sigma^n \subset \RR^{n+1}$ is a closed self-shrinker (for any $n$) and $\Sigma$ is not round, then
there is a nearby hypersurface $\Gamma$ with the following properties:
\begin{enumerate}
\item $\lambda (\Gamma) < \lambda (\Sigma)$.
\item $\Gamma$ is inside of $\Sigma$, i.e., the compact  region of $\RR^{n+1}$ bounded by $\Sigma$ contains $\Gamma$.
\item  $ \left( H - \frac{1}{2} \, \langle x , \nn \rangle \right) > 0$ on $\Gamma$.
\end{enumerate}
\end{Lem}

\begin{proof}
Let $\Sigma_s$ be a one-parameter family of normal graphs given by
\begin{equation}
	\Sigma_s= \{ x + s \, u(x) \, \nn (x) \, | \, x \in \Sigma \} \, ,
\end{equation}
where $\nn(x)$ is the outward unit normal to $\Sigma$.  Equation ($4.10$) in \cite{CM1} computes that
\begin{align}	\label{e:410cm1}
\frac{d}{ds} \, \big|_{s=0} \,  \left( H - \frac{\langle x , \nn \rangle}{2} \right) 
= - L \, u \, .
\end{align}
The following three facts are proved in \cite{CM1}{\footnote{See \cite{CM1} corollary $5.15$ and theorem $4.30$, theorem $0.15$, and   equation ($4.10$), respectively.}}:
\begin{itemize}
\item The lowest eigenfunction $u$ for $L$ is positive and has $L \, u = \mu \, u$ for some $\mu > 1$.
\item There exists $\epsilon > 0$, so that for every $s$ with $0< |s| < \epsilon$ we have $\lambda (\Sigma_s ) < \lambda (\Sigma)$.
Here $\Sigma_s$ is the graph of $s \, u$.
\item With $s$ as above, $\Sigma_s$ has
\begin{align}
	\left( H - \frac{1}{2} \, \langle x , \nn \rangle \right) < 0 {\text{ if }} s > 0 \, , \\
	\left( H - \frac{1}{2} \, \langle x , \nn \rangle \right) > 0 {\text{ if }} s < 0 \, .
\end{align}
\end{itemize}
Thus, we see that for $s \in (-\epsilon , 0)$, $\Gamma = \Sigma_s$ has the three properties.
\end{proof}

\section{Rescaled MCF}

A one-parameter family of hypersurfaces $M_t$ evolves by rescaled MCF if it satisfies
\begin{align}
	\partial_t x = - \left( H - \frac{1}{2} \, \langle x , \nn \rangle \right)  \, \nn \, .
\end{align}
The rescaled MCF is equivalent to MCF, up to rescalings in space and a reparameterization of time.  Namely, if $\Sigma_t$ is a MCF, then 
  $t\to \Sigma_{-\e^{-t}}/\sqrt{\e^{-t}}$ is a solution to the rescaled MCF equation and vice versa. 
Moreover, rescaled MCF is the negative gradient flow for the $F$ functional and self-shrinkers are the fixed points for this flow.

\subsection{Mean convex rescaled MCF}

The next  ingredient is a result from \cite{IW} about ``mean convex rescaled MCF'' where  
\begin{equation}	\label{e:mcv}
	\left( H - \frac{1}{2} \, \langle x , \nn \rangle \right)  \geq 0 \, .
\end{equation}
We will often refer to this quantity as the rescaled mean curvature.

This result is the rescaled analog of that mean convexity is preserved for MCF.

\begin{Lem}	\label{l:mcvx1}
Let $M_t \subset \RR^{n+1}$ be a smooth rescaled MCF of closed hypersurfaces for $t \in [0,T]$. 
\begin{enumerate}
\item If 
\eqr{e:mcv} holds on $M_0$, then it also holds on $M_t$ for every $t \in [0,T]$. 
\item If in addition $\left( H - \frac{1}{2} \, \langle x , \nn \rangle \right) >0$ at least at one point of $M_0$,
then $\left( H - \frac{1}{2} \, \langle x , \nn \rangle \right) >0$ on $M_t$ for all $t>0$.
\end{enumerate}
  In particular, the flow is monotone in that 
$M_t$ is inside $M_s$ for $t > s$.
\end{Lem}

The key to proving this is a Simons type equation for rescaled MCF
\begin{align}	\label{e:simonsrmcfH}
\left( \partial_t - L \right) \,   \left( H - \frac{\langle x , \nn \rangle}{2} \right) 
=   0 \, ,
\end{align}
 which  follows from \eqr{e:410cm1} with
  $u = - \left( H - \frac{\langle x , \nn \rangle}{2} \right)  $.

\begin{proof}[Proof of Lemma \ref{l:mcvx1}]
Set $\phi = H - \frac{1}{2} \, \langle x , \nn \rangle$.  Equation \eqr{e:simonsrmcfH} is equivalent to that
\begin{align}	\label{e:starting}
	\left( \partial_t - \Delta \right) \, \phi = \left( |A|^2 + \frac{1}{2} \right) \, \phi 
	- \langle \frac{x}{2} , \nabla \phi \rangle \, .
\end{align}
Since    $\phi (x,0) \geq 0$ by assumption, the parabolic maximum principle gives $\phi (x,t) \geq 0$ for all $t\geq 0$, giving (1).  The parabolic strong maximum principle then gives (2).  The last claim follows immediately since $\nn$ is the outward pointing normal.
\end{proof}

 \section{Simons type identity for the rescaled $A$}  Using covariant derivatives, 
 the operators $\partial_t$ and $L$ can be extended to tensors (see Hamilton, \cite{Ha}; cf. lemma $10.8$ in \cite{CM1}).
 We will next show that the   Simons type identity \eqr{e:simonsrmcfH} is in reality the trace of a tensor equation:

\begin{Lem}	\label{l:simonsA}
If $M_t$ is flowing by rescaled MCF and $
	L  = \Delta - \langle \frac{x}{2} , \nabla \cdot \rangle + |A|^2 + \frac{1}{2} 
$, then  
\begin{align}	\label{e:simonsrmcf}
\left( \partial_t - L \right) \,  A
&=  - A \, , \\
\left( \partial_t - L \right) \,  H
&=  - H \, , \\
\left( \partial_t - L \right) \,  \langle x , \nn \rangle 
&=  - 2 \, H  \, , \\
(\partial_t-L)\, \e^t\left(A-\frac{\langle x,\nn\rangle}{2n}\, \, g \right)&=\frac{\e^t}{n}\,\left(H-\frac{\langle x,\nn\rangle}{2}\right)\,g \, .
\end{align}
We will later use that the first equation is equivalent to that $\left( \partial_t - L \right) \,  \e^t\,A=0$.  
\end{Lem}

\begin{proof}
Given a hypersurface and an orthonormal frame $\{ e_i \}$, we follow the convention in \cite{CM1} by
setting
\begin{equation}
	a_{ij} = A(e_i , e_j) = \langle \nabla_{e_i} e_j , \nn \rangle \, ,
\end{equation}
so that $H = - \sum_{i=1}^n a_{ii}$.
 For a general hypersurface,   the Laplacian of $A$ is (see, e.g., lemma B.8 in \cite{Si} where the sign convention for the $a_{ij}$'s is reversed)
 \begin{equation}	\label{e:simsim}
 	  \left( \Delta  A \right)_{ij}  = -    |A|^2 \, a_{ij} - H \, a_{ik} \, a_{jk} - H_{ij} \, .
 \end{equation}
 Here $H_{ij}$ is the $ij$ component of the Hessian of $H$.
 
 If a hypersurface evolves by $\partial_t x = f \, \nn$ for a function $f$, then we 
 have the following standard formulas for the variations of the components of the 
 metric $g_{ij}$, the components of the second fundamental form $a_{ij}$, and the unit normal $\nn$
 (see, e.g., lemmas $7.4$ and $7.6$ in \cite{HP})
 \begin{align}
 	\partial_t \, g_{ij} &= - 2 \, f \, a_{ij} \, , \\
	\partial_t \nn &= - \nabla f \, , \\
	\partial_t \, a_{ij} &= f_{ij} - f \, a_{ik} \, a_{kj} \, .
\end{align}
In the rescaled MCF, we have $f=   \langle x , \nn \rangle/2 - H$.

 We will extend $\partial_t$ to tensors by using the covariant derivative $\nabla_{\partial_t}$
  in the metric $g_{ij} \times dt^2$
 (cf. \cite{Ha} or the appendix in \cite{W4}).  
 If we let $e_i$ be the evolving frame on $M_t$ coming from pushing forward the frame on $M_0$, 
 then the   formula for the Christoffel symbols gives  
 \begin{equation}
 	\partial_t e_i = \frac{1}{2} \, g^{jk} (\partial_t \, g_{ij}) \, e_k = - f \, g^{jk} \,  a_{ij} \, e_k 
	  \,  .
 \end{equation}
 Using the Leibniz rule, 
 we have
 \begin{align}
 	\partial_t g_{ij} = \partial_t \, \left( g (e_i , e_j ) \right) = \left( \partial_t g \right) (e_i , e_j) 
	+ g ( \partial_t e_i , e_j) + g (\partial_t e_j , e_i) \, ,
 \end{align}
 so we see that this choice makes $\partial_t g = 0$.
 
 Similarly, working at a point where the $e_i$'s are orthonormal, the Leibniz rule gives that
 \begin{align}
 	\left( \partial_t \, A \right) (e_i , e_j) &= \partial_t a_{ij} - A (\partial_t e_i , e_j) - A (e_i , \partial_t e_j) \notag \\
	&= \left(  f_{ij}  - f \, a_{ik} a_{kj}  \right) - 2 \, \left( - f \, a_{ik} \, a_{kj} \right) 
	= f_{ij} + f \, a_{ik} \, a_{kj} 
	\, .
 \end{align}
 Combining this with \eqr{e:simsim}, we compute the heat operator on $A$
 \begin{align}	\label{e:heatA}
 	\left( \partial_t \, A - \Delta A \right) (e_i , e_j) &=
	f_{ij} + f \, a_{ik} \, a_{kj}  + |A|^2 \, a_{ij} + H \, a_{ik} a_{kj} + H_{ij} \notag \\
	&= \left( f + H \right)_{ij} +   \left( f + H \right) \, a_{ik} \, a_{kj} 
	+ |A|^2 \, a_{ij}  \\
	&= \frac{1}{2} \, \langle x , \nn \rangle_{ij} +  \frac{1}{2} \, \langle x , \nn \rangle \, a_{ik} \, a_{kj} 
	+ |A|^2 \, a_{ij} \, . \notag
 \end{align}
Rewriting this in terms of the $L$ operator (using $x^T$ for the tangential part of $x$) gives
   \begin{align}	\label{e:heatA2}
  	\left( \partial_t \, A - L A \right) (e_i , e_j) &= 
	\left( \partial_t \, A - \Delta A \right) (e_i , e_j)  + \left( \frac{1}{2} \, \nabla_{x^T} A - \frac{1}{2} \, A - |A|^2 \, A\right) (e_i , e_j) \notag \\
		&= \frac{1}{2} \, \left\{ \langle x , \nn \rangle_{ij} +   \langle x , \nn \rangle \, a_{ik} \, a_{kj} 
		+ \left( \nabla_{x^T} A \right)(e_i , e_j) -  a_{ij}
		 \right\}
	  \, .
 \end{align}
 
 The next step is to compute the operators on $\langle x , \nn \rangle$.  To simplify the calculation, let
  $e_i$ be an orthonormal frame and work at a point where  its tangential covariant derivatives vanish.  
 In particular, at this point we have $\nabla_{e_i} e_j = a_{ij} \, \nn$ and $\nabla_{e_i} \nn = - a_{ij} \, e_j$. 
 Using this and the general fact that $\nabla_{e_i} x = e_i$
  (cf. the proof of lemma
  $5.5$ in \cite{CM1}),  we get at this point  
   \begin{align}
 	 \langle x , \nn \rangle_i &= \nabla_{e_i} \langle x , \nn \rangle = \langle e_i , \nn \rangle + 
	 \langle x , \nabla_{e_i} \nn \rangle = 
	- A (x^T , e_i)   \, .
\end{align}
To compute the Hessian, we will differentiate this.    Using first the Leibniz rule, then the fact that 
$\nabla A$ is fully symmetric in all three inputs by the Codazzi equations and the fact that
\begin{equation}
	\langle \nabla_{e_j} x^T , e_k \rangle = \langle \nabla_{e_j} x  , e_k \rangle - \langle \nabla_{e_j} x^{\perp} , e_k \rangle
	= \delta_{jk} + a_{jk} \, \langle x , \nn \rangle \, ,
\end{equation}
we get that
\begin{align}
	 \langle x , \nn \rangle_{ij} &= - \left( \nabla_{e_j} A \right) (x^T , e_i) - A ( \left( \nabla_{e_j} x^T \right)^T , e_i )
	 - A ( x^T , \nabla_{e_j} e_i ) \notag \\
	 &= - \left( \nabla_{x^T} A \right) (e_i , e_j) - A(e_i , e_j) - a_{ij} \, a_{kj} \, \langle x , \nn \rangle \, , \label{e:coda}
 \end{align}
   Using \eqr{e:coda} in \eqr{e:heatA2} gives
 the first claim.  The second claim follows from the first since traces commute with covariant derivatives ($g$ is parallel, even with respect to time).  The third claim follows from the second claim and the Simons equation
 \eqr{e:simonsrmcfH}  for $H - \frac{1}{2} \, 
 \langle x , \nn \rangle$.
 
  For the last claim, we use that $ \left( \partial_t  - L  \right)\, \e^t \, A = 0$ by the first claim 
  and, 
   since the metric is time-parallel, the third claim gives that 
 \begin{align}
 	\e^{-t} \, \left( \partial_t  - L  \right)\,   \left( \e^t \, \langle x , \nn \rangle \, g \right) =
	\langle x , \nn \rangle \, g +  \left( 
	\left( \partial_t  - L  \right)\,   \langle x , \nn \rangle \right) \, g = \left( \langle x , \nn \rangle  - 2 \, H \right) \, g \, .
 \end{align}
 
\end{proof}

\subsection{Bounding $A$}  In this subsection,
 we use the matrix equation in combination with the parabolic maximum principle to bound $A$ for the rescaled MCF by the rescaled mean curvature when the latter is nonnegative.  

To achieve this bound, we will need some simple computations.  First we will need the quotient rule.  Let $f$ be a symmetric $2$-tensor and $h$ a positive function
\begin{align}	\label{e:quotient}
		(\partial_t-\Delta ) \, \frac{f }{h} &= \frac{(\partial_t-\Delta) \, f }{h} - \frac{f (\partial_t-\Delta) \,h}{h^2} 
		+ 2 \, \frac{  \nabla_{\nabla h} f   }{h^2} -2\,  \frac{f \, |\nabla h|^2}{h^3} \\
		&= \frac{(\partial_t-\Delta) \, f }{h} - \frac{f (\partial_t-\Delta) \,h}{h^2} 
		+ \frac{2}{h} \,  \nabla_{ \nabla h} \left( \frac{f}{h} \right)
	\, .\notag
\end{align}
Next let $V$ be a vector field, and set $\cL\,f=\Delta\,f+  \nabla_V  f $; we will also use that
\begin{align}
(\partial_t-\cL)\,|f|^2&=2\, \langle f ,(\partial_t-\cL)\,f \rangle - 2 \, |\nabla f|^2\leq 2\, \langle f ,(\partial_t-\cL)\,f \rangle \, ,\\
(\partial_t-\cL) \,\left| \frac{f }{h} \right|^2&\leq 2\, \langle \frac{f}{h} ,(\partial_t-\cL)\,\frac{f}{h} \rangle \, ,\\
(\partial_t-\cL ) \, \frac{f }{h} &=\frac{(\partial_t-\cL) \, f }{h} - \frac{f (\partial_t-\cL) \,h}{h^2} 
		+ \frac{2}{h} \,  \nabla_{ \nabla h} \left( \frac{f}{h} \right) \, .
\end{align}
In particular, if $L\,f=\cL\,f+K\,f$ for some function $K$, $(\partial_t-L)\,f=0$, and $(\partial_t-L)\,h=0$, then the two previous equations give
\begin{align}
(\partial_t-\cL ) \, \left|\frac{f }{h}\right|^2 &\leq 2 \, \langle \frac{f}{h} ,(\partial_t-\cL)\,\frac{f}{h} \rangle
= \frac{4}{h} \, \langle \nabla_{ \nabla h} \left( \frac{f}{h} \right) , \frac{f}{h} \rangle \notag \\
&= \frac{2}{h} \, \langle \nabla \left|\frac{f}{h}\right|^2, \nabla h \rangle=2\, \langle \nabla \left|\frac{f}{h}\right|^2, \nabla \log h \rangle\, .
\end{align}

Applying the above to $f=\e^t\,A$ and $h=H-\frac{\langle x,\nn\rangle}{2}$ and using Lemma \ref{l:simonsA} yields the following differential inequality for the ratio 
\begin{equation}
B=\frac{f}{h}=\frac{\e^t\,A}{H-\frac{\langle x,\nn\rangle}{2}}\, :
\end{equation}

\begin{Lem}  \label{l:diffineq}
If $M_t\subset \RR^{n+1}$ are hypersurfaces flowing by the rescaled MCF, $H-\frac{\langle x,\nn\rangle}{2}>0$ on the initial hypersurface, and $B$ is as above, then
\begin{align}
(\partial_t-\cL) \, |B|^2\leq 2\,\langle \nabla |B|^2,\nabla \log \left(H-\frac{\langle x,\nn\rangle}{2}\right)\rangle \, .
\end{align}
\end{Lem}

The parabolic maximum principle now implies the following:

\begin{Cor}	\label{c:theabovec}
If $M_t\subset \RR^{n+1}$ are closed hypersurfaces flowing by the rescaled MCF with $H-\frac{\langle x,\nn\rangle}{2}>0$ on the initial hypersurface, then
\begin{equation}
\left| A \right|^2\leq C\, \e^{-2t}\left|H-\frac{\langle x,\nn\rangle}{2}\right|^2\, ,
\end{equation}
for some constant $C$ depending on the initial hypersurface.  
\end{Cor}

\begin{proof}
This follows from the parabolic maximum principle and Lemma \ref{l:diffineq}.  
\end{proof}

\subsection{Curvature bounds for blow ups}

The next lemma establishes that $H$ bounds $|A|$ on the regular part of any multiplicity one tangent flow when the initial hypersurface is closed and has positive rescaled mean curvature.  In particular, any such tangent flow has $H \geq 0$.

\begin{Lem}	\label{l:curvbd}
Let $M_t\subset \RR^{n+1}$ be closed hypersurfaces flowing by the rescaled MCF with $H-\frac{\langle x,\nn\rangle}{2}\geq 0$ on the initial hypersurface $M_0$.  There exists $C>0$ so that if  $T_{t}$ is a tangent flow at the first singular time and $T_{-1}$ is multiplicity one, 
then:
\begin{itemize}
\item  $|A| \leq C \, H$  on the regular set $\reg (T_{-1})$.
\end{itemize}
\end{Lem}

\begin{proof}
Let $\tau$ be the first singular time and  $y \in M_{\tau}$ the singular point.   
 
By Lemma \ref{l:mcvx1}, every $M_t$ has
% \begin{equation}	\label{e:posH}
 	$H - \frac{1}{2} \, \langle x , \nn \rangle \geq 0$,
%\end{equation}
so  the flow is nested and  
\begin{equation}	\label{e:C1}
	\frac{1}{2} \, \left| \langle x , \nn \rangle \right| \leq \frac{|x|}{2} \leq C_1 \equiv \frac{1}{2} \,  \max_{ M_0} |x| \, .
\end{equation}
Corollary \ref{c:theabovec} gives a constant $C_0>0$ depending on the initial hypersurface $M_0$
so that
\begin{equation}	\label{e:fromcor}
\left| A \right| \leq C_0\, \e^{-t}\left( H-\frac{\langle x,\nn\rangle}{2}\right)
\leq C_0 \, \left( H-\frac{\langle x,\nn\rangle}{2}\right) \leq
C_0 \, H + C_0 \, C_1 \, .
\end{equation}

Fix a compact subset $\Omega \subset \reg (T_{-1})$.  Since $\Omega$ is smooth and multiplicity one, White's version of Brakke's regularity theorem, \cite{W2} (cf. Allard's theorem; see \cite{Al} or \cite{Si}),
 gives a sequence $h_i \to 0$ so that (a subset of)
\begin{equation}
	\Sigma_i \equiv \frac{1}{h_i} \, \left( M_{\tau - h_i^2} - y \right)
\end{equation}
smoothly converges to $\Omega$.  Let $A_i$ and $H_i$ be the second fundamental form and mean curvature, respectively, of $\Sigma_i$.  It follows from \eqr{e:fromcor} that
\begin{equation}
	 \left| A_i \right| = h_i \, |A| \leq h_i \, \left( C_0 \, H + C_0 \, C_1\right) = C_0  \, H_i + h_i \, C_0 \, C_1 \, .
\end{equation}
Since we have smooth convergence to $\Omega$ and $h_i \to 0$, we can pass to limits to get that
\begin{equation}
	\left| A_{\Omega} \right| \leq C_0 \, H_{\Omega} \, .
\end{equation}
The lemma follows since $\Omega$ is arbitrary. 
\end{proof}

\begin{Rem}
A similar statement holds for blow ups at the first singular time.
\end{Rem}

\section{Finite time blow up for rescaled MCF}

The next result shows that there is finite time blow up for solutions of the rescaled MCF when the initial closed 
hypersurface has a strict positive lower bound for $  H - \frac{1}{2} \, \langle x , \nn \rangle $.

\begin{Lem}	\label{l:sing}
Given $c> 0$, there exists $T_c$ so that
if $M_0$ is a closed hypersurface in $\RR^{n+1}$ with 
$  H - \frac{1}{2} \, \langle x , \nn \rangle   \geq c$, then the rescaled MCF $M_t$ hits a singularity for $t \leq T_c$.
\end{Lem}

\begin{proof}
Set $\phi =  H - \frac{1}{2} \, \langle x , \nn \rangle  $.    
Let $m(t)$ be the minimum of $\phi$ at time $t$.  By \eqr{e:starting}, we have
\begin{equation}
	\left( \partial_t - \Delta \right) \, \e^{- \frac{t}{2} } \, \phi \geq - \langle \frac{x}{2} , \nabla \e^{- \frac{t}{2} } \, \phi \rangle \, , 
\end{equation}
so the parabolic maximum principle gives that the minimum of $\e^{- \frac{t}{2} } \, \phi$ is non-decreasing in time.  This 
  gives exponential growth of $m(t)$, i.e., 
\begin{equation}	\label{e:expgrowth}
	m(t) \geq \e^{ \frac{t}{2} } \, m(0) \geq  \e^{ \frac{t}{2} } \,  c > 0 \, .
\end{equation}
   Since $M_t$ lies in a bounded set by Lemma \ref{l:mcvx1} and $|\nn|=1$, we have a constant $C$ so that
   $\left| \frac{1}{2} \, \langle x , \nn \rangle \right| \leq |x| / 2 \leq C$ for all $t \geq 0$.  
   Combining this with \eqr{e:expgrowth}, we can choose $T_1 > 0$ so that
   if $t \geq T_1$, then 
  $
   	H \geq \frac{\phi}{2}$.  
	In particular, when $t \geq T_1$, this means that
\begin{equation}
	n \, |A|^2 \geq H^2 \geq  \frac{ \phi^2}{4} \, .
\end{equation}
Using this in \eqr{e:starting}, we get  for $t \geq T_1$ that
\begin{equation}
	\left( \partial_t - \Delta \right) \, \phi = \left( |A|^2 + \frac{1}{2} \right) \, \phi - \frac{1}{2} \, \langle x , \nabla \phi \rangle
	\geq \frac{1}{4n} \, \phi^3 - \frac{1}{2} \, \langle x , \nabla \phi \rangle \, .
\end{equation}
Using this, the parabolic maximum principle gives 
the differential inequality (for $t \geq T_1$)
\begin{equation}
	m' (t) \geq \frac{1}{4n} \, m^3(t) \, .
\end{equation}
Comparing this with 
  the  ODE $f'(t) \geq \frac{1}{4n} \, f^3(t)$, this gives finite time blow up for $\phi$ which implies the finite time blow up for $|A|$, giving the desired singularity.
\end{proof}

\section{Tangent cones to self-shrinkers}

In this section, we will  study  two types of $n$-dimensional rectifiable integral varifolds.  The first are weak solutions of the self-shrinker equation \eqr{e:ss}; 
these are  critical points for the $F$ functional and are called $F$-stationary.   The second are stationary with respect to the Euclidean volume and will simply be called stationary;  these    arise as blow ups of the first.

The singular set and regular set of a rectifiable varifold $\Sigma$ are denoted by $\sing (\Sigma)$ and $\reg (\Sigma)$, respectively.

The main result of this section is the following:

\begin{Pro}	\label{p:tc2}
If $\Sigma \subset \RR^{n+1}$ is a $F$-stationary rectifiable varifold, $\lambda (\Sigma)<\frac{3}{2}$, and
 there is a constant $C > 0$ so that
\begin{equation}	\label{e:constC}
	|A| \leq C \, H {\text{ on the regular set }} \reg (\Sigma) \, ,
\end{equation}
then $\Sigma$ is smooth.
\end{Pro}

We will use a blow up analysis, where we analyze the tangent cones, to prove the proposition.   It follows from \cite{Al} (cf.
section $42$ in \cite{Si})
that an integral $n$-rectifiable varifold has stationary integral rectifiable tangent cones as long as the generalized mean curvature $H$ is locally in $L^p$ for some $p> n$.  This is trivially satisfied for stationary varifolds where $H = 0$ and  for $F$-stationary varifolds where $H$ is locally in $L_{\infty}$.

\begin{Lem}	\label{l:tc1}
Suppose that $\Sigma \subset \RR^{n+1}$ is a $F$-stationary  rectifiable varifold satisfying \eqr{e:constC}.
If $\Gamma$ is any multiplicity one blow up of $\Sigma$, then $\Gamma$ is   stationary    and
$|A| \equiv 0$ on $\reg (\Gamma)$.
\end{Lem}

\begin{proof}
By definition, there are sequences $s_i \to 0$ and   $y_i \to y \in \RR^{n+1}$ so that
\begin{equation}
	\Sigma_i \equiv \frac{1}{s_i} \, \left( \Sigma - y_i \right)
\end{equation}
converges with multiplicity one to $\Gamma$.   Hence, the regular part of each $\Sigma_i$ satisfies
$
	|A_i| \leq C \, H_i 
$
with the  constant $C$ from \eqr{e:constC}
since this inequality is scale invariant.
Since  $H = \frac{1}{2} \, \langle x , \nn \rangle$ is locally bounded on $\Sigma$,  
the limit $\Gamma$ is stationary.

Since $\Gamma$ is multiplicity one, Allard's theorem{\footnote{Since the sequence of rescalings
has  locally bounded $H$, Allard gives uniform local $C^{1,\alpha}$ estimate graphical estimates.  Elliptic theory then gives uniform estimates on higher derivatives and, thus, smooth convergence.}}
   (see \cite{Al} or \cite{Si}) implies that the $\Sigma_i$'s converge smoothly on the 
regular set $\reg (\Gamma)$ and, thus,  
\begin{equation}
	|A_{\Gamma}| \leq C \, H_{\Gamma} {\text{ on }} \reg (\Gamma) \, .
\end{equation}
Since $\Gamma$ is stationary, we have $H=0$ on $\reg (\Gamma)$ and the lemma follows.
\end{proof}

Tangent cones are limits of rescalings about a fixed point.  An iterated tangent cone is  a tangent cone to a tangent cone to a tangent cone$\dots$  (with  finitely many iterations).

\begin{Cor}	\label{c:tc2}
Let $\Sigma \subset \RR^{n+1}$ be a $F$-stationary  rectifiable varifold 
satisfying \eqr{e:constC}.
 If $\lambda (\Sigma) < 2$, then every iterated tangent cone $\Gamma$ is stationary, has $\lambda (\Gamma)\leq \lambda (\Sigma)$, has multiplicity one, and satisfies 
$|A| \equiv 0$ on $\reg (\Gamma)$.
\end{Cor}

\begin{proof}
This follows   from Lemma \ref{l:tc1} since iterated tangent cones can be realized as blow ups by taking a diagonal sequence
and any limit has multiplicity one since  $\lambda (\Sigma) < 2$.
\end{proof}

We will need two elementary lemmas from blow up analysis:

\begin{Lem}	\label{l:elGa}
Suppose that $\Gamma \subset \RR^{n+1}$ is an $n$-dimensional stationary rectifiable varifold with $|A| \equiv 0$ on $\reg (\Gamma)$.
If $\Gamma$ is a cone, $\sing (\Gamma) \subset \{ 0 \}$ and $n > 1$, then $\sing (\Gamma) = \emptyset$.
\end{Lem}

\begin{proof}
Since $|A|\equiv 0$ on the regular set,  the closure of each connected component of $\reg (\Gamma)$ is contained in a hyperplane.  Since $0$ is the only singular point and $\Gamma$ is a cone,  the closure of each connected component of  $\reg (\Gamma)$ must in fact be a hyperplane through $0$.  However, since two distinct hyperplanes  in $\RR^{n+1}$ with $0$  in their closure always intersect away from $0$ (since $n > 1$),   $\Gamma$   consists  of a single hyperplane through the origin and thus is, in particular, smooth, though of course it could have multiplicity.  
\end{proof}

The  above  lemma does not extend to $n=1$; in particular, three half-lines meeting at $0$ with angles $\frac{2\pi}{3}$ is a stationary configuration.  Thus something else, like an entropy bound, is needed to rule out such a singularity.    Also, as noted, it could have multiplicity.  

\vskip2mm
The second elementary lemma that we will need is often implicitly used in Federer type dimension reduction arguments.    It is the following:

\begin{Lem}	\label{l:dimred}
Suppose that $\Gamma \subset \RR^{n+1}$ is a stationary rectifiable varifold and $\Gamma$ is a cone.
If $y \in \sing (\Gamma) \setminus \{ 0 \}$, then every tangent cone to $\Gamma$ at $y$ is of the form
$\Gamma' \times \RR_y$ where $\RR_y$ is the line in the direction $y$ and $\Gamma'$ is a conical stationary varifold in $\RR^n$.
\end{Lem}

\begin{proof}[Proof of Proposition \ref{p:tc2}]
By Allard's regularity theorem, it is enough to show that every tangent cone to $\Sigma$ is a multiplicity one hyperplane.  We will prove this by contradiction.  Suppose, therefore, that $x\in \sing (\Sigma)$, $\Gamma_1$ is a tangent cone to $\Sigma$ at $x$, and $\Gamma_1$ is not a multiplicity one hyperplane.

By assumption,  $\lambda (\Sigma) < 3/2$ 
and  $\Sigma$ satisfies \eqr{e:constC}; therefore, Corollary \ref{c:tc2} implies that $\Gamma_1$ is stationary, satisfies $|A| \equiv 0$ on $\reg (\Gamma_1)$, and
\begin{equation}
	\lambda (\Gamma_1) \leq \lambda (\Sigma) < 3/2 \, .
\end{equation}
  Since $\Gamma_1$ is not a hyperplane ($x$ is   singular), Lemma \ref{l:elGa} 
  gives a singular point $y_1 \in \sing (\Gamma_1) \setminus \{ 0 \}$.  Thus, 
Lemma \ref{l:dimred}  gives a multiplicity one tangent cone $\Gamma_2$ to $\Gamma_1$ at $y_1$ with
\begin{equation}
	\Gamma_2 = \Gamma_2' \times \RR_{y_1} \, , 
\end{equation}
where $\Gamma_2'$ is a stationary cone in $\RR^n$.  Since $y$ was a singular point, $\Gamma_2'$ is not a hyperplane.
Furthermore, since Corollary \ref{c:tc2} applies to all iterated tangent cones, we know that $\Gamma_2$, and thus also $\Gamma_2'$ have
$|A| \equiv 0$ on their regular parts.

We can now repeat the argument to get a stationary cone $\Gamma_3' \subset \RR^{n-1}$ that is multiplicity one, is not a hyperplane, and also has $|A| \equiv 0$ on its regular part.  

Repeating this $n-1$ times eventually gives a stationary cone $\Gamma_{n}' \subset \RR^2$ that is not a line and where $\Gamma_{n}' \times \RR^{n-1}$ is an iterated tangent cone for $\Gamma$.  In particular, since entropy is preserved under products with $\RR$, we must have that
\begin{equation}
	\lambda (\Gamma_n') < 3/2 \, .
\end{equation}
This implies that $\Gamma_n'$ consists of at most two rays through the origin.  However, the only such configuration that is stationary is when there are exactly two rays meeting at angle $\pi$ to form a line.  This contradiction completes the proof.
\end{proof}

\section{Low entropy singularities for mean convex rescaled MCF}

An important tool in this paper is a classification of singularities for rescaled MCF starting from a closed 
rescaled mean convex hypersurface with low entropy.  Here, ``rescaled mean convex'' means that \eqr{e:mcv} holds.
This classification will be given in Corollary \ref{c:higher} below.  The first step in this direction is the following regularity theorem and partial classification:
 
\begin{Pro}	\label{p:higher}
Let $M_t$ be a rescaled MCF of closed hypersurfaces in $\RR^{n+1}$,
satisfying \eqr{e:mcv} on $[0,T)$ and $\lambda (M_0) < 3/2$.  
If there is a singularity at $(y,T)$, then there is a tangent flow $T_t$ with:
\begin{itemize}
\item $T_{-1}$ is a smooth, embedded, self-shrinker with $\lambda (T_{-1}) < 3/2$ and multiplicity one.
\item $T_{-1}$ has $H \geq 0$ and is not flat.
\end{itemize}
\end{Pro}

 \begin{proof} 
The blow up argument given in lemma $8$ of \cite{I1} (cf. \cite{W1}, \cite{W3}) gives a tangent flow $T_t$ so that $T_{-1}$ is an $F$-stationary rectifiable varifold.
  Huisken's monotonicity and the properties of the entropy in \cite{CM1} (see lemma $1.11$ and section $7$ there) give that
\begin{align}
 \lambda (T_{-1}) \leq \lambda (M_0) < 3/2 \, , 
\end{align}
so that $T_{-1}$ has multiplicity one and is embedded.
Hence,  by Lemma \ref{l:curvbd}, 
\begin{itemize}
\item $|A| \leq C \, H$  on the regular set $\reg (T_{-1})$.
\end{itemize}
Finally,   Proposition \ref{p:tc2}
gives that
 $M_{-1}$ is smooth.  
 \end{proof}

We will combine Proposition \ref{p:higher} with the following 
  classification of smooth, embedded mean convex self-shrinkers in arbitrary dimension from theorem $0.17$ in
  \cite{CM1}:

 \begin{Thm}[\cite{CM1}]	\label{t:huisken}
 $\SS^k\times \RR^{n-k}$ are the only smooth complete embedded self-shrinkers without boundary, with polynomial volume growth, and   $H \geq 0$ in $\RR^{n+1}$.
  \end{Thm}

  The $\SS^k$ factor in Theorem \ref{t:huisken} is round and has radius $\sqrt{2k}$; we allow the possibilities of a 
  sphere ($n-k = 0$) or a
  hyperplane (i.e., $k=0$), although Brakke's theorem rules out the multiplicity one 
  hyperplane as a tangent flow at a singular point.

The classification of  smooth embedded self-shrinkers with $H \geq 0$ began with
  \cite{H1},  where Huisken showed that round spheres are the  only closed ones. In   \cite{H2}, Huisken  showed that 
   generalized cylinders $\SS^k\times \RR^{n-k}\subset \RR^{n+1}$ are the only open ones with
  polynomial volume growth  and $|A|$ bounded.    Theorem $0.17$ in \cite{CM1} 
 completed the classification by  removing the $|A|$ bound.

\vskip2mm
Combining Proposition \ref{p:higher} with Theorem \ref{t:huisken} gives:

\begin{Cor}	\label{c:higher}
Let $M_t$ be a rescaled MCF of closed hypersurfaces in $\RR^{n+1}$,
satisfying \eqr{e:mcv} on $[0,T)$ and $\lambda (M_0) < 3/2$.  
If there is a singularity at $(y,T)$, then there is a multiplicity one tangent flow  $T_t$  
of the form  $\SS^k\times \RR^{n-k}$ for some $k >0$.
\end{Cor}

Combining this with the monotonicity of the entropy  will give the following lower bound for entropy and topological rigidity for rescaled mean convex hypersurfaces:

\begin{Pro}	\label{p:rigid}
Let $M^n$ be a  closed hypersurface    
satisfying \eqr{e:mcv}.  If $\lambda (M) <  \min \{  \lambda (\SS^{n-1}) , \frac{3}{2} \}$, then
$M$  is diffeomorphic 
to $\SS^n$ and
$\lambda (M) \geq   \lambda (\SS^n)$.
\end{Pro}

\begin{proof}
We can assume that $H - \frac{1}{2} \, \langle x , \nn \rangle > 0$ at   least at one point.
To see this, suppose instead that
$H - \frac{1}{2} \, \langle x , \nn \rangle \equiv 0$,  so that $M$ is a self-shrinker.  If $M = \SS^n$, then we are done.  Otherwise,  
Lemma \ref{l:perturb} gives a nearby graph  
 $\tilde{M}$ over $M$ with $H - \frac{1}{2} \, \langle x , \nn \rangle > 0$  and
$\lambda (\tilde{M}) < \lambda (M)$.

Let $M_t$ be the rescaled MCF with $M_0 = M$, so that     Lemma \ref{l:mcvx1} gives
  \begin{equation}	\label{e:post}
 	H - \frac{1}{2} \, \langle x , \nn \rangle > 0 {\text{ for all }} t > 0  \, .
\end{equation}
Lemmas \ref{l:sing} and \ref{l:mcvx1} give a singularity in finite time   inside of $M$.  Thus, Corollary \ref{c:higher} gives a multiplicity one tangent flow $T_t$ 
at this point of the form
	$ \SS^k\times \RR^{n-k}$  for some $k >0$.
  By the monotonicity of the entropy under MCF (lemma $1.11$ in \cite{CM1}), its invariance under dilation, and its lower semi-continuity under limits, we have
\begin{equation}
	\lambda (T_{-1}) \leq \lambda (M) < \min \{  \lambda (\SS^{n-1}) , 3/2 \} \, .
\end{equation}
By \cite{St}, we have $\lambda (\SS^{n}) <  \lambda (\SS^{n-1} \times \RR) < \dots$, so we conclude that $k=n$ and 
$
	\lambda (\SS^{n}) \leq \lambda (M)$. 
	 Finally,   White's version, \cite{W2},  of Brakke's   theorem implies that $T_{-1}$ is the smooth limit of rescalings of the $M_t$'s.  In particular, $M$ itself is diffeomorphic to $T_{-1} = \SS^n$.
 
 \end{proof}

\section{Proof of the main theorem}

 \begin{proof} 
 [Proof of Theorem \ref{t:minentropy}]
Let $\Sigma$ be a closed self-shrinker in $\RR^{n+1}$
with  $\lambda (\Sigma)< \min \{  \lambda (\SS^{n-1} ) , 3/2 \}$.
By
  Proposition \ref{p:rigid}, $\tilde{\Sigma}$ is diffeomorphic to $\SS^n$ and
$\lambda (\tilde{\Sigma})  \geq \lambda (\SS^n)$.

We will show that there is a gap, i.e.,   some $\epsilon > 0$ so that $\lambda (\Sigma)  \geq \lambda (\SS^n) + \epsilon$ if $\Sigma$ is not round.
 Suppose instead that there is a sequence of closed self-shrinkers $\Sigma_i \ne \SS^n$ with
\begin{equation}
	\lambda (\SS^n) \leq \lambda (\Sigma_i) < \lambda (\SS^n) + 2^{-i} \, .
\end{equation}
By Huisken, \cite{H1}, none of the $\Sigma_i$'s is strictly convex  since they are not round.

Perturbing the $\Sigma_i$'s with Lemma \ref{l:perturb} and then applying rescaled MCF to the perturbations gives a sequence of rescaled MCF's $\tilde{M}_{i,t}$ so that
\begin{itemize}
 \item Each initial hypersurface $\tilde{M}_{i,0}$ is not strictly convex and has $\lambda (\tilde{M}_{i,0}) < \lambda (\Sigma_i)$.
 \item 
 Each $\tilde{M}_{i,t}$ has a multiplicity one spherical singularity in finite time.
 \end{itemize}
 The second claim follows from  Lemma \ref{l:sing}.
 
 We can now create a new sequence of rescaled MCF's $M_{i,t}$ by rescaling the $\tilde{M}_{i,t}$'s about the spherical singularity so that the new flows
 satisfy:
 \begin{enumerate}
\item Each $M_{i,t}$
converges smoothly to the round sphere as $t \to \infty$.
 \item Each initial hypersurface ${M}_{i,0}$ is a $C^2$ graph over $\SS^n$ with $C^2$ norm exactly  $\epsilon > 0$.
 \item Every hypersurface    ${M}_{i,t}$ for $t \geq 0$ is a $C^2$ graph over $\SS^n$ with $C^2$ norm at most  $\epsilon$.

 \end{enumerate}
 By (3), we can choose a subsequence that converges smoothly to a limiting rescaled MCF $M_t$.     Now consider the $F$ functional along $M_t$.  By (1), we have
 \begin{align}
 	\lim_{t \to \infty} F(M_t) = \lambda (\SS^n) \, .
\end{align}
On the other hand, we know that
\begin{align}
	F( M_0) \leq \lambda (M_0) \leq \liminf \lambda (\Sigma_i ) = \lambda (\SS^n) \, .
\end{align}
Since the rescaled MCF is the gradient flow for $F$, we see that the flow must be constant and, thus, every $M_t$ is round.  This contradicts (2) which says that the initial hypersurfaces are strictly away from the round sphere, completing the proof.

\end{proof}


\begin{thebibliography}{A}

\bibitem[AbL]{AbL}
U. Abresch and J. Langer, 
{\emph{The normalized curve shortening flow and homothetic solutions}}, JDG   23 (1986),   175--196.

\bibitem[Al]{Al}
W.K Allard, \emph{On the first variation of a varifold}. Ann. of Math. (2) 95 (1972), 417--491.

\bibitem[A]{A}
S.B. Angenent,
\emph{Shrinking doughnuts}, In: Nonlinear diffusion equations and their equilibrium states,
Birkh\"auser, Boston-Basel-Berlin, 3, 21-38, 1992.


\bibitem[B]{B}
K. Ball,  
{\emph{The reverse isoperimetric problem for Gaussian measure}}, Discrete Comput. Geom. 10 (1993),   411--420. 


\bibitem[Br]{Br}
K. Brakke,  {\emph{The motion of a surface by its mean curvature}}. Mathematical Notes, 20. Princeton University Press, Princeton, N.J., 1978.

\bibitem[Ch]{Ch}
D. Chopp,
\emph{Computation of self-similar solutions for mean curvature flow}.
Experiment. Math. 3 (1994), no. 1, 1--15.
 
\bibitem[CM1]{CM1}
T.H. Colding and W.P. Minicozzi II,
\emph{Generic mean curvature flow I; generic singularities},  Annals of Math., 
175 (2012), 755--833.

\bibitem[CM2]{CM2}
\bysame, 
\emph{Generic mean curvature flow II; dynamics of a closed smooth singularity}, 
in preparation.

\bibitem[CM3]{CM3}
\bysame, \emph{Smooth compactness of self-shrinkers}, Comm. Math. Helv., 87 (2012),
463--487.

 \bibitem[GaHa]{GaHa}  
M. Gage and R. Hamilton, 
\emph{The heat equation shrinking convex plane curves}, 
JDG 23 (1986), 69-96.

\bibitem[G]{G} 
M. Grayson, 
\emph{The heat equation shrinks embedded plane curves to round points}, JDG 26 (1987), 285-314.

\bibitem[Ha]{Ha}
R. S. Hamilton,
 {\it Harnack estimate for the mean curvature flow}. JDG  41 (1995),   215--226.



 \bibitem[H1]{H1}
G. Huisken,
\emph{Asymptotic behavior for singularities of the mean curvature flow}.
JDG  31 (1990),  285--299.

\bibitem[H2]{H2}
\bysame,
\emph{Local and global behaviour of hypersurfaces moving by mean curvature}.
Differential geometry: partial differential equations on manifolds (Los Angeles, CA, 1990), 175--191,
Proc. Sympos. Pure Math., 54, Part 1, Amer. Math. Soc., Providence, RI, 1993.

 \bibitem[HP]{HP}
 G.  Huisken and A. Polden,  
{\it Geometric evolution equations for hypersurfaces}. Calculus of variations and geometric evolution problems, 45--84, 
Lecture Notes in Math., 1713, Springer, Berlin, 1999. 

\bibitem[HS1]{HS1} 
G. Huisken and C. Sinestrari, 
\emph{Mean curvature flow singularities for mean convex surfaces},  
Calc. Var. Partial Differ. Equ. 8 (1999) 1--14.

\bibitem[HS2]{HS2} 
\bysame,
\emph{Convexity estimates for mean curvature flow and singularities of mean convex surfaces}. Acta Math. 183 (1999), no. 1, 45--70.

 \bibitem[I1]{I1}
T. Ilmanen,
\emph{Singularities of Mean Curvature Flow of Surfaces}, preprint, 1995, \\
http://www.math.ethz.ch/\~{}/papers/pub.html.

\bibitem[I2]{I2}
\bysame, \emph{Lectures on Mean Curvature Flow and Related Equations}, 
(Trieste Notes), 1995.

\bibitem[IW]{IW}
T. Ilmanen and B. White, in preparation.

\bibitem[K]{K}  
D.M. Kane. \emph{The Gaussian Surface Area and Noise Sensitivity of
Degree-d Polynomial Threshold Functions}, in Proceedings of the 25th
annual IEEE Conference on Computational Complexity (CCC 2010), pp.
205-210.

\bibitem[KDS]{KDS}  
A. Klivans, R. O'Donnell, R. Servedio. \emph{Learning geometric concepts via
Gaussian surface area}. In Proc. 49th IEEE Symposium on Foundations of Computer Science FOCS '08.

\bibitem[KKM]{KKM}
N. Kapouleas, S. Kleene, and N.M. M\"oller, \emph{Mean curvature self-shrinkers of high genus:
non-compact examples}, preprint, http://arxiv.org/pdf/1106.5454.

%\bibitem[MiSi]{MiSi}
%J. H. Michael and L. M. Simon,
%{\it Sobolev and mean-value inequalities on generalized submanifolds of $\RR^n$}.
%Comm. Pure Appl. Math. 26 (1973), 361--379.

\bibitem[M]{M} 
N.M. M\"oller, \emph{Closed self-shrinking surfaces in $\RR^3$ via the torus}, http://arxiv.org/pdf/1111.7318v1.pdf.

\bibitem[Na]{Na}  
F. Nazarov, \emph{On the maximal perimeter of a convex set in $\RR^n$ with
respect to a Gaussian measure}.
In Geometric aspects of functional analysis (2001-2002), pages
169Ð187. Lecture Notes in Math., Vol. 1807, Springer, 2003.

\bibitem[N]{N} 
X.H. Nguyen, 
\emph{Construction of Complete Embedded Self-Similar Surfaces under Mean Curvature Flow. Part III}, 
preprint, http://arxiv.org/pdf/1106.5272v1.

\bibitem[Si]{Si}
L. M. Simon,   
\emph{Lectures on Geometric Measure Theory}, 
Proc.
of the CMA, ANU No. 3,  Canberra, 1983.

\bibitem[St]{St}
A. Stone, 
\emph{A density function and the structure of singularities of the mean curvature flow}, 
Calc. Var. 2 (1994), 443--480.

\bibitem[W1]{W1}
B. White, 
\emph{Stratification of minimal surfaces, mean curvature
flows, and harmonic maps}. 
J. Reine Angew. Math. 488 (1997), 1--35.

\bibitem[W2]{W2}
\bysame,
\emph{A local regularity theorem for mean curvature flow}. 
Ann. of Math. (2) 161 (2005),  1487--1519.

\bibitem[W3]{W3}
\bysame,
{\emph{The size of the singular set in mean curvature flow of
mean-convex sets}}.  J. Amer. Math. Soc.  13  (2000),   
665--695.

\bibitem[W4]{W4}
\bysame,
{\emph{The nature of singularities in mean curvature flow of mean-convex sets}, 
J. Amer. Math. Soc. 16 (2003)}, 123--138.

  \end{thebibliography}
\end{document}